\documentclass[11pt]{amsart}

\usepackage{amssymb}
\usepackage{amsmath}
\usepackage{mathrsfs}
\usepackage{amsfonts}
\usepackage{amsthm}
\usepackage{graphicx}
\usepackage[colorlinks=true, pdfstartview=FitV, linkcolor=blue,citecolor=blue, urlcolor=blue]{hyperref}
\usepackage{verbatim}

\usepackage[utf8]{inputenc}
\usepackage[T1]{fontenc}
\numberwithin{equation}{section}

\theoremstyle{definition}

\theoremstyle{plain}
\newtheorem{theorem}{Theorem}

\newtheorem{lem}{Lemma}

\newtheorem{prop}{Proposition}

\numberwithin{equation}{section}

\renewcommand{\r}{\mathbb{R}}

\renewcommand{\and}{\quad\textrm{ and }\quad}

\let \epsilon\varepsilon
\let \phi\varphi

\newcommand{\dVol}{\mathrm{dVol}}
\newcommand{\tr}{\operatorname{tr}}

\author{Tomasz Cie\'{s}lak}
\address{Institute of Mathematics \newline Polish Academy of Sciences \newline \'Sniadeckich 8, 00-656 Warszawa, Poland}
\email{cieslak@impan.pl}

\author{Micha\l{} Gaczkowski}
\address{Faculty of Mathematics and Information Sciences, Warsaw University of Technology,
Ul. Koszykowa 75, 00-662 Warsaw, Poland}
\email{Michal.Gaczkowski@pw.edu.pl}

\author{Wojciech Kry\'nski}
\address{Institute of Mathematics \newline Polish Academy of Sciences \newline \'Sniadeckich 8, 00-656 Warszawa, Poland}
\email{krynski@impan.pl}

\title[A functional inequality between Hessians]{A functional inequality between Hessians in spaces with non-zero curvature}

\begin{document}
\begin{abstract}
A version of the recent functional inequality (see \cite{CHFS})  between the Hessians of the square root and the logarithm of positive functions  is proven in spaces with non-zero curvature.
\end{abstract}

\maketitle

\section{Introduction.}
The main result of this note is an extension of the inequality stated in the flat case in \cite{CHFS} to the setting of Riemannian manifolds.
\begin{theorem}\label{glowne}
Let $(M,g)$ be a compact Riemannian manifold with no boundary. There exists a positive constant $C$ such that for any positive smooth function $u\colon M\to \r$ the following inequality holds
\begin{equation}\label{main}
\int_M |\nabla^2 \sqrt{u}|^2\;\dVol_g\leq C \int_M u|\nabla^2 \log u|^2\;\dVol_g,
\end{equation}
where $\nabla$ is the covariant derivative and $\dVol_g$ is the volume form on $M$.
\end{theorem}
Before emphasizing the potential meaning of the theorem, let us briefly comment on the proof. The proof in the flat case \cite{CHFS} is based on pointwise identities (see \cite[page 62]{CHFS}), which necessarily contain additional metric related components in the Riemannian case.
One could try to replace the above identities by its tensorial version, but this would require giving the tensorial meaning to some non-trivial terms.
Instead, we use the Bochner formula, which is recalled in the next section and leads to terms that can be integrated by parts in a clear way.

As a by-product of our reasoning, we also obtain the following Riemannian version of the Bernis-type inequality.
The proof of this inequality is a straightforward adaptation of the flat case argument from \cite{Winkler} to the setting of Riemannian geometry.
\begin{prop}\label{Bernis}
Let $(M,g)$ be a compact Riemannian manifold with no boundary. There exists a positive constant $C$ such that for any positive smooth function $u\colon M\to \r$ the following inequality holds
\begin{equation}
\int_M \frac{|\nabla u|^4}{u^3}\;\dVol_g\leq C \int_M u|\nabla^2 \log u|^2\;\dVol_g.
\end{equation}
\end{prop}
The version of inequality \eqref{main} in the flat case has proven to be a fruitful tool in the investigation of the system of PDEs arising in 1D thermoelasticity problems, see \cite{BC1,BC2}. The global existence of unique solutions was established in \cite{BC1}, and a full characterization of the asymptotic states of a heated string was given in \cite{BC2}. Furthermore, an application of the inequality to the combustion problems was given in \cite{S.Li}. The applicability of \eqref{main} to systems of PDEs comes from the fact that the right-hand side of \eqref{main} coincides exactly with the negative of the dissipation term of the evolution of the Fisher information along the heat flow. One can thus construct functionals based on Fisher information that can be bounded using \eqref{main} and the dissipation term.

In the context of the Ricci flow, the Fisher information was used in a pointwise sense by Hamilton in \cite{Hamilton}, and later, Perelman introduced his functional \emph{W} in \cite{Perelman}
\[
\emph{W}(g_{ij}, f, \tau)=\int_M \left[\tau (|\nabla f|^2+R)+f-n\right](4\pi \tau)^{-\frac{n}{2}}e^{-f}\;\dVol_g,
\]
where, $n=\dim M$, $g$ is a Riemannian metric tensor on $M$, $R$ is the scalar curvature of $g$,  $f$ is a function on $M$ satisfying
\[
\int_M (4\pi \tau)^{-\frac{n}{2}}e^{-f}\;\dVol_g=1,
\]
and $\tau$ is a rescaled time. Let us take a closer look at the functional in order to see its relation to the Fisher information.
Along the Ricci flow the dissipation of $\emph{W}$ is given by
\[
\frac{d{\emph W}}{dt}=2\tau \left|Ric_{ij}+\nabla_i\nabla_j f-\frac{1}{2\tau}g_{ij}\right|^2(4\pi \tau)^{-\frac{n}{2}}e^{-f}\;\dVol_g.
\]
Notice that for  $u:=e^{-f}$, the functional ${\emph W}$ resembles the Fisher information of $u$ plus some geometric contribution terms, while
the dissipation term plays the role of $u|\nabla^2 \log u|^2$ plus some geometric contribution terms. Hence, one may expect that our inequality \eqref{main} can be used in the analysis of the Ricci flow.

\section{Preliminaries.}
As stated in Theorem \ref{glowne}, we shall find a relation between $\int_{M}|\nabla ^2 \sqrt{u}|^2 \dVol_g$ and $\int_{M}u|\nabla ^2 \log{u}|^2\dVol_g$, where $M$ is a compact manifold with no boundary, $\log$ is the natural logarithm and $|\cdot|^2$ is understood as the norm of an endomorphism. For a given $A\colon TM\to TM$, it is defined as the trace $\tr\left(g^{-1}\circ A^* \circ g \circ A\right)$, where the metric tensor $g$ is regarded as a map $g\colon TM\to T^*M$ that raises indices, and $A^*\colon T^*M\to T^*M$ is the dual map. Note that the matrix of $A^*$ is given as $A^T$ in any frame. Consequently, in an orthonormal frame,  $|A|^2$ becomes the Frobenius norm of the matrix of the endomorphism $A$ and is computed as $\tr A^TA$.

Recall that the gradient $\nabla f$ of a function $f\colon M\to \r$ is a vector field on $M$ defined via the identity $g(\nabla f,\cdot)=df$. Then, the Hessian can be understood as an endomorphism $\nabla^2f\colon TM\to TM$ of the tangent bundle defined as $X\mapsto \nabla^{LC}_X(\nabla f)$, where $\nabla^{LC}$ stands for the Levi-Civita connection. Note that the composition $g\circ\nabla^2 f$ is the bi-linear form on $M$ defined as $\nabla^{LC}(df)$, where now $\nabla^{LC}$ is the Levi-Civita connection acting on 1-forms. Indeed, this is a consequence of the fact that the Levi-Civita connection is compatible with the metric. Later on, we shall abuse the notation and treat $\nabla^2 f$ as a bilinear form, writing $\nabla^2 f(X, Y):=\nabla^{LC}_X(df)(Y)=g(X,(\nabla^2 f)(Y))$, where $X$ and $Y$ are vector fields on $M$.

\begin{lem}\label{auxi}
Let $u\colon M\to \r$. Then the following identity holds
\[
g(\nabla u,\nabla\left|\nabla u\right|^2)=2\nabla^2u(\nabla u,\nabla u).
\]
\end{lem}
\begin{proof}
First, by definition, $\nabla\left|\nabla u\right|^2$ is the vector field that is $g$-dual to the differential 1-form $d(g(\nabla u,\nabla u))$. Hence, $g(\nabla u,\nabla\left|\nabla u\right|^2)=d(g(\nabla u,\nabla u))(\nabla u)=(\nabla u)(g(\nabla u,\nabla u))$. Since the Levi-Civita connection is compatible with the metric $g$, we obtain $(\nabla u)(g(\nabla u,\nabla u))=2g(\nabla^{LC}_{\nabla u}\nabla u, \nabla u)=2g(\nabla^2u(\nabla u),\nabla u)$, which is nothing but $2\nabla^2u(\nabla u,\nabla u)$.
\end{proof}

The pointwise Bochner formula, which holds for every function $u\colon M\to \r$, takes the form
\begin{equation}\label{Bochner}
\frac{1}{2} \Delta_g |\nabla u|^2=g\left(\nabla u, \nabla \Delta_g u\right)+|\nabla ^2 u|^2+Ric(\nabla u, \nabla u),
\end{equation}
where the Laplacian $\Delta_g(f)$ is defined as $\tr\nabla^2f$. We shall also use the following formulas, which can be verified by a direct calculation using covariant derivatives
\begin{equation}\label{log}
\Delta_g \log{u}=\frac{\Delta_gu}{u}-\frac{|\nabla u|^2}{u^2},
\end{equation}
and
\begin{equation}\label{sqrt}
\Delta_g \sqrt{u}=\frac{1}{2}\frac{\Delta_gu}{\sqrt{u}}-\frac{|\nabla u|^2}{4u^{3/2}}.
\end{equation}
Moreover, for a compact Riemannian manifold $(M,g)$ without boundaries, one has the expected integration by parts formula, see \cite[Theorem 2.2.7]{Stavrov}
\begin{equation}\label{byparts}
\int_M g(\nabla f,\nabla u)\;\dVol_g=\int_Mf\Delta_gu\;\dVol_g
\end{equation}
for any functions $u,f\colon M\to\r$.

Finally, let us recall a fact that is obvious in the flat case but requires a short proof in the general Riemannian setting.
\begin{lem}
Let $f\colon M\to \r$, $\dim M=n$. Then, the following pointwise estimate
\begin{equation}\label{raz}
|\Delta_g f|^2\leq n|\nabla^2 f|^2
\end{equation}
holds.
\end{lem}
\begin{proof}
The formula is a consequence of the identity $(\tr A)^2\leq n\tr (A^2)$ which is true for any symmetric $n\times n$-matrix $A$. It is convenient to work in an orthonormal frame $(X^1,\ldots,X^n)$. Then the metric $g$ becomes the identity matrix, and the entries of the Hessian can be computed as $(\nabla^{LC}_{X_i}df)(X_j)$. Hence, it is sufficient to show that we have $(\nabla^{LC}_{X_i}df)(X_j)=(\nabla^{LC}_{X_j}df)(X_i)$.
In fact, for any vector fields $X$ and $Y$ and a 1-form $\alpha$ the following holds
\[
\begin{aligned}
(\nabla^{LC}_X\alpha)(Y)&=X(\alpha(Y))-\alpha(\nabla_X^{LC}Y)\\
&=X(\alpha(Y))-\alpha(\nabla^{LC}_YX)-\alpha([X,Y])\\
&=X(\alpha(Y))-Y(\alpha(X))-\alpha([X,Y])+(\nabla^{LC}_Y\alpha)(X)
\end{aligned}
\]
where we used the Leibniz rule as well as the torsion-freeness of the Levi-Civita connection. Now, for $\alpha=df$ we have $X(\alpha(Y))-Y(\alpha(X))=[X,Y](f)$ which completes the proof.
\end{proof}

\section{Proof of Theorem \ref{glowne}.}
Let us calculate $\int_{M}|\nabla ^2 \sqrt{u}|^2 \;\dVol_g$ first. We use the Bochner formula \eqref{Bochner}. We emphasize here that it is crucial to exchange the Hessian by objects which undergo the integration by parts in a convenient way. Laplacian $\Delta_g$ is one of such objects, see for instance \cite{Stavrov}. The Bochner formula applied to $\sqrt{u}$ and \eqref{sqrt} yield
\begin{eqnarray*}
&&\int_{M}|\nabla ^2 \sqrt{u}|^2 \;\dVol_g\\
&=&\int_{M}\frac{1}{2} \Delta_g |\nabla \sqrt{u}|^2 \dVol_g -\int_{M}g\left(\nabla \sqrt{u}, \nabla \Delta_g \sqrt{u}\right) \dVol_g-\int_{M} Ric(\nabla \sqrt{u}, \nabla \sqrt{u}) \dVol_g\\
&=&-\frac{1}{4}\int_{M}\frac{1}{u} g\left(\nabla u, \nabla \Delta_g u\right) \dVol_g+\frac{1}{8}\int_{M}\frac{|\nabla u|^2}{u^2}\Delta_g u \;\dVol_g\\
&+&\frac{1}{4}\int_{M} g\left(\frac{\nabla u}{\sqrt{u}},\nabla \left(\frac{|\nabla u|^2}{u^{3/2}}\right)\right) \dVol_g-\frac{1}{4}\int_{M} \frac{Ric(\nabla u,\nabla u)}{u} \dVol_g.
\end{eqnarray*}
Hence, using Lemma \ref{auxi}, we obtain
\begin{eqnarray}\label{pierwsze}
&&\int_{M}|\nabla ^2 \sqrt{u}|^2 \dVol_g\nonumber \\
&=&-\frac{1}{4}\int_{M}\frac{1}{u} g\left(\nabla u, \nabla \Delta_g u\right)\dVol_g+\frac{1}{8}\int_{M}\frac{|\nabla u|^2}{u^2}\Delta_g u \dVol_g\nonumber \\
&+&\frac{1}{2}\int_{M} \frac{\nabla^2 u\left(\nabla u,\nabla u\right)}{u^2}\dVol_g-\frac{3}{8}\int_{M}\frac{|\nabla u|^4}{u^3}\dVol_g-\frac{1}{4}\int_{M} \frac{Ric(\nabla u,\nabla u)}{u} \dVol_g.\nonumber \\
\end{eqnarray}
Moreover, we notice that integration by parts, Lemma \ref{auxi} and \eqref{log} give
\begin{equation}\label{drugie}
\int_{M}\frac{|\nabla u|^2}{u^2}\Delta_g u \;\dVol_g=-2\int_{M}\frac{\nabla^2 u\left(\nabla u,\nabla u\right)}{u^2}\;\dVol_g
+2\int_{M}\frac{|\nabla u|^4}{u^3}\dVol_g.
\end{equation}
Plugging \eqref{drugie} into \eqref{pierwsze}, we obtain
\begin{eqnarray}\label{trzecie}
&&\int_{M}|\nabla ^2 \sqrt{u}|^2\dVol_g\\
&=&-\frac{1}{4}\int_{M}\frac{1}{u} g\left(\nabla u, \nabla \Delta_g u\right)\dVol_g+\frac{1}{4}\int_{M} \frac{\nabla^2 u\left(\nabla u,\nabla u\right)}{u^2}\dVol_g\nonumber \\
&-&\frac{1}{8}\int_{M}\frac{|\nabla u|^4}{u^3}\dVol_g-\frac{1}{4}\int_{M} \frac{Ric(\nabla u,\nabla u)}{u}\dVol_g.\nonumber
\end{eqnarray}
On the other hand, a similar approach, exploiting \eqref{log} and the Bochner formula applied to $\log{u}$, the application of Lemma \ref{auxi} and integration by parts \eqref{byparts}, gives
\begin{eqnarray*}
&&\int_{M}u|\nabla ^2 \log{u}|^2 \dVol_g=\\
&&\frac{1}{2}\int_{M}u\Delta_g|\nabla \log u|^2\dVol_g-\int_M ug\left(\nabla \log u,\nabla \Delta_g\log u\right)\dVol_g-\int_{M} u Ric(\nabla \log u,\nabla \log u) \dVol_g\\
&=&-\frac{1}{2}\int_{M} g\left(\nabla u,\nabla |\nabla \log u|^2\right)\dVol_g-\int_{M}\frac{1}{u} g\left(\nabla u, \nabla \Delta_g u\right)\dVol_g \\
&+&\int_{M}\frac{|\nabla u|^2}{u^2}\Delta_g u \dVol_g
+\int_{M} g\left(\nabla u,\nabla \left(\frac{|\nabla u|^2}{u^2}\right)\right)\dVol_g
-\int_{M} \frac{Ric(\nabla u,\nabla u)}{u}\dVol_g \\
&=&\int_{M} \frac{\nabla^2 u\left(\nabla u,\nabla u\right)}{u^2}\dVol_g-\int_{M}\frac{|\nabla u|^4}{u^3}\dVol_g \\
&-&\int_{M}\frac{1}{u} g\left(\nabla u, \nabla \Delta_g u\right)\dVol_g \\
&+&\int_{M}\frac{|\nabla u|^2}{u^2}\Delta_g u \;\dVol_g -\int_{M} \frac{Ric(\nabla u,\nabla u)}{u} \dVol_g.
\end{eqnarray*}
Applying \eqref{drugie} in the above, we obtain
\begin{eqnarray}\label{czwarte}
&&\int_{M}u|\nabla ^2 \log{u}|^2 \dVol_g\\
&=&-\int_{M}\frac{1}{u} g\left(\nabla u, \nabla \Delta_g u\right)\dVol_g-\int_{M} \frac{\nabla^2 u\left(\nabla u,\nabla u\right)}{u^2}\dVol_g\nonumber \\
&+&\int_{M}\frac{|\nabla u|^4}{u^3}\dVol_g-\int_{M} \frac{Ric(\nabla u,\nabla u)}{u} \dVol_g.\nonumber
\end{eqnarray}
Comparing \eqref{trzecie} i \eqref{czwarte}, we arrive at
\begin{eqnarray}\label{piate}
&&\int_{M}|\nabla ^2 \sqrt{u}|^2 \;\dVol_g\\
&=&\frac{1}{4}\int_{M}u|\nabla ^2 \log{u}|^2 \dVol_g+\frac{1}{2}\int_M\frac{\nabla^2 u\left(\nabla u,\nabla u\right)}{u^2}\dVol_g\nonumber\\
&-&\frac{3}{8}\int_{M}\frac{|\nabla u|^4}{u^3}\dVol_g.\nonumber
\end{eqnarray}
Next, we estimate $\int_{M}\frac{|\nabla u|^4}{u^3}\dVol_g$ using again Lemma \eqref{auxi} and integration by parts \eqref{byparts} (the argument below is also the proof of Proposition \ref{Bernis})
\begin{eqnarray*}
&&\int_{M}\frac{|\nabla u|^4}{u^3}\dVol_g=\int_{M}g(\nabla u, \nabla \log u)|\nabla \log u|^2\dVol_g\\
&=&-\int_{M}u\Delta_g \log u|\nabla \log u|^2\dVol_g-2\int_{M}u \nabla^2\log u \left(\nabla \log u,\nabla \log u\right) \dVol_g\\
&\leq&\left(\int_{M}u\left(\Delta_g \log u\right)^2\dVol_g\right)^{\frac{1}{2}}\left(\int_{M}\frac{|\nabla u|^4}{u^3}\dVol_g\right)^{\frac{1}{2}}\\
&+&2\left(\int_{M}u|\nabla^2 \log u|^2\dVol_g\right)^{\frac{1}{2}}\left(\int_{M}\frac{|\nabla u|^4}{u^3}\dVol_g\right)^{\frac{1}{2}}.
\end{eqnarray*}
Consequently, applying \eqref{raz} to $f=\log u$, and next dividing both sides of the above inequality by $\left(\int_{M}\frac{|\nabla u|^4}{u^3}\dVol_g\right)^{1/2}$, we arrive at
\begin{equation}\label{szoste}
\int_{M}\frac{|\nabla u|^4}{u^3}\dVol_g\leq C \int_{M}u|\nabla^2 \log u|^2\dVol_g,
\end{equation}
and the proof of Proposition \ref{Bernis} is finished.

Next, we notice that by \eqref{drugie}
\[
\int_M\frac{\nabla^2 u\left(\nabla u,\nabla u\right)}{u^2}\dVol_g=\int_{M}\frac{|\nabla u|^4}{u^3}\dVol_g-\frac{1}{2}\int_M\frac{|\nabla u|^2}{u^2}\Delta_g u \;\dVol_g.
\]
Since (see \eqref{log})
\[
\int_M\frac{|\nabla u|^2}{u^2}\Delta_g u \;\dVol_g=\int_M \Delta_g \log u \frac{|\nabla u|^2}{u}\;\dVol_g+\int_{M}\frac{|\nabla u|^4}{u^3}\dVol_g,
\]
we observe, utilizing again \eqref{raz} and \eqref{szoste}, that
\begin{eqnarray}\label{ostatnie}
&&\int_M\frac{\nabla^2 u\left(\nabla u,\nabla u\right)}{u^2}\dVol_g=\frac{1}{2}\int_{M}\frac{|\nabla u|^4}{u^3}\dVol_g-\frac{1}{2}\int_M \Delta_g \log u \frac{|\nabla u|^2}{u}\;\dVol_g\nonumber\\
&\leq& \frac{1}{2}\int_{M}\frac{|\nabla u|^4}{u^3}\dVol_g+\frac{1}{2}\left(\int_{M}u\left(\Delta_g \log u\right)^2\dVol_g\right)^{\frac{1}{2}}\left(\int_{M}\frac{|\nabla u|^4}{u^3}\dVol_g\right)^{\frac{1}{2}}\nonumber\\
&\leq& C\int_{M}u|\nabla^2 \log u|^2\dVol_g.\nonumber\\
\end{eqnarray}
Plugging \eqref{szoste} and \eqref{ostatnie} in \eqref{piate}, we obtain
\[
\int_{M}|\nabla ^2 \sqrt{u}|^2 \;\dVol_g\leq \left(\frac{1}{4}+C\right)\int_{M}u|\nabla ^2 \log{u}|^2 \dVol_g.
\]

\end{document}